\documentclass[12 pt,reqno]{amsart}
\usepackage{amscd,amssymb,amsthm}
\usepackage[arrow,matrix]{xy}
\usepackage{graphicx}
\usepackage{cite}
\oddsidemargin=0.3in \evensidemargin=0.3in \theoremstyle{plain}
\newtheorem{thm}[subsection]{Theorem}

\newtheorem{cor}[subsection]{Corollary}

\theoremstyle{definition}

\newtheorem{definition}[subsection]{Definition}

\newtheorem{pb}[subsection]{Problem}

\numberwithin{equation}{section} 

\begin{document}
\title [\texttt{Extremal Triangular Chain Graphs for Bond Incident Degree (BID) Indices}]{Extremal Triangular Chain Graphs for Bond Incident Degree (BID) Indices}
\author{Akbar Ali$^{\dag,\ddag}$ and Akhlaq Ahmad Bhatti$^{\dag}$}
 \address{$^{\dag}$Department of Mathematics \\ National University of Computer and Emerging Sciences, B-Block, Faisal Town, Lahore, Pakistan.}
 \address{$^{\ddag}$Department of Mathematics \\ University Of Gujrat, Gujrat, Pakistan.}

\email{akbarali.maths@gmail.com,akhlaq.ahmad@nu.edu.pk}



\begin{abstract}

A general expression for calculating the bond incident degree (BID) indices of certain triangular chain graphs is derived. The extremal triangular chain graphs with respect to several well known BID indices are also characterized over a particular collection of triangular chain graphs.
\end{abstract}
\maketitle
\section{\textbf{Introduction}}


All the graphs considered in the present study are finite, undirected, simple and connected. The vertex set and edge set of a graph $G$ will be denoted by $V(G)$ and $E(G)$ respectively. Undefined notations and terminologies from (chemical) graph theory can be found in \cite{EsBo2013,Ha69,Tr92}.

According to Todeschini and Consonni \cite{ToCo2000} ``\textit{molecular descriptor} is the final result of a logical and mathematical procedure which transforms chemical information encoded within a symbolic representation of a molecule into an useful number or the result of some standardized experiment''. A molecule can be represented by a graph in which the vertices correspond to atoms while the edges represent covalent bonds between atoms of the molecule \cite{2}.  A graph-based molecular descriptor is simply known as a \textit{topological index} \cite{EsBo2013}. In graph theoretical notation, topological indices are the numerical parameters of a graph which are invariant under graph isomorphisms. These indices are often used to model the physicochemical properties of chemical compounds in quantitative structure-property relation (QSPR) and quantitative structure-activity relation (QSAR) studies \cite{De99,d2,f1,Tr92}. A considerable amount of such indices can be represented as the sum of edge contributions of the graph \cite{VuDe10,Vu10}. Furthermore, in many cases these edge contributions depend exclusively on the degrees of end vertices of the observed edge \cite{VuDu11}. These kinds of indices are called \textit{bond incident degree indices} \cite{VuDu11} (these indices form a subclass of the class of all (vertex) degree based topological indices \cite{d4,f5,g1,r1}) whose general form \cite{hollas,VuDe10} is:
\begin{equation}\label{z}
TI=TI(G)=\displaystyle\sum_{uv\in E(G)}f(d_{u},d_{v})=\displaystyle\sum_{1\leq a\leq b\leq \Delta(G)}x_{a,b}(G).\theta_{a,b} \ ,
\end{equation}
where $uv$ is the edge connecting the vertices $u$ and $v$ of the graph $G$, $d_{u}$ is the degree of the vertex $u$, $E(G)$ is the edge set of $G$, $\Delta(G)$ is the maximum degree in $G$, $\theta_{a,b}$ is a non-negative real valued function depending on $a,b$, and $x_{a,b}(G)$ is the number of edges in $G$ connecting the vertices of degrees $a$ and $b$. Many well known topological indices such as the Randi$\acute{c}$ index \cite{r3}, atom-bond connectivity index \cite{e3}, sum-connectivity index \cite{z2}, first geometric-arithmetic index \cite{v1}, augmented Zagreb index \cite{f4}, Albertson index \cite{a1}, harmonic index \cite{f2}, second Zagreb index \cite{g22}, modified second Zagreb index \cite{n1} and natural logarithm of multiplicative sum Zagreb index \cite{e1} are the special cases of (\ref{z}). Moreover, a certain class of bond incident degree (BID) indices was proposed and several indices from this class were examined for their chemical applicability in \cite{VuDe10}. Besides, many indices of the form (\ref{z}) exist in the literature, for example see the review \cite{g1}, articles \cite{AA1,AA3,d4,f5,gutman,g2,VuDe10,VuDu11} and related references cited therein.

A $k$-polygonal system is a connected geometric figure obtained by concatenating congruent regular $k$-polygons side to side in a plane in such a way that the figure divides the plane into one infinite (external) region and a number of finite (internal) regions, and all internal regions must be congruent regular $k$-polygons. For $k=3,4,6$, the $k$-polygonal system corresponds to triangular animal \cite{g12}, polyominoes \cite{g11}, benzenoid system \cite{gutman89} respectively. In a $k$-polygonal system, two polygons are said to be adjacent if they share a side. The characteristic graph (or dualist or inner dual) of a given $k$-polygonal system consists of vertices corresponding to $k$-polygons of the system; two vertices are adjacent if and only if the corresponding $k$-polygons are adjacent. A $k$-polygonal system whose characteristic graph is the path graph is called $k$-polygonal chain. A $k$-polygonal chain can be represented by a graph (called $k$-polygonal chain graph) in which the edges represent sides of a polygon while the vertices correspond to the points where two sides of a polygon meet. In the rest of paper, by the term \textit{$k$-polygonal chain} we actually mean \textit{$k$-polygonal chain graph}.


In recent years, a significant amount of research has been devoted to solve the problems of finding closed form formulae for different BID indices of $4$-polygonal (polyomino) chains and characterizing the extremal polyomino chains with respect to the aforementioned indices over the set of all polyomino chains with fixed number of squares (for example, see \cite{AA2,AA3,An2,deng,z} and related references cited therein). In addition, extremal $5$-polygonal (pentagonal) and $6$-polygonal (benzenoid) chains for various well known BID indices were recently characterized in \cite{AA4} and \cite{r1} respectively. Furthermore, for the $k$-polygonal chains (where $k\geq7$), the above mentioned problems are rather easy to solve and hence a simple solution can be provided. However, for $k=3$ the problems under consideration are not easy and thence it is natural to attempt these problems. We will attempt these problems for a certain collection of triangular chains. To state these problems in a more precise way, we need to recall some basic definitions and terminologies concerning triangular chains. A triangular chain in which every vertex has degree at most four is called \textit{linear triangular chain}. A graph $H$ is said to be a subgraph of the graph $G$ if $V(H)\subseteq V(G)$ and $E(H)\subseteq E(G)$. A subgraph $H$ of graph $G$ is said to be induced subgraph if whenever $u,v\in V(H)$ and $uv\in E(G)$ then $uv\in E(H)$. A part of a triangular chain $T_{n}$ is said to be \textit{segment} if it forms an induced subgraph which is the maximal linear triangular chain in $T_{n}$. A triangle in a triangular chain is said to be \textit{terminal} (respectively \textit{nonterminal}) if it is adjacent with only one (respectively two) other triangle(s). A segment containing terminal triangle(s) is called \textit{terminal segment} and a segment which is not terminal is known as \textit{nonterminal segment}. Let $T_{n}$ has $s$ segments $S_{1}, S_{2},S_{3},...,S_{s}$. The number of triangles in a segment $S_{i}$ (where $1\leq i\leq s$) is called its \textit{length} and is denoted by $l(S_{i})$ (or simply by $l_{i}$ for the sake of brevity). An $s$-tuple $(a_{1},a_{2},...,a_{s})$ is said to be \textit{length vector} of $T_{n}$ if and only if $a_{i}=l_{i}$ for all $i$ where $1\leq i\leq s$. If $(a_{1},a_{2},...,a_{s})$ is a length vector of $T_{n}$ and $s\geq3$ then we assume that $a_{1},a_{s}$ are the lengths of terminal segments. All the segments of a certain triangular chain are shown in the Figure \ref{f0} and the length vector of this chain is $(6,5,4,3)$.
\begin{figure}
    \centering
    \includegraphics[width=1.6in, height=1.45in]{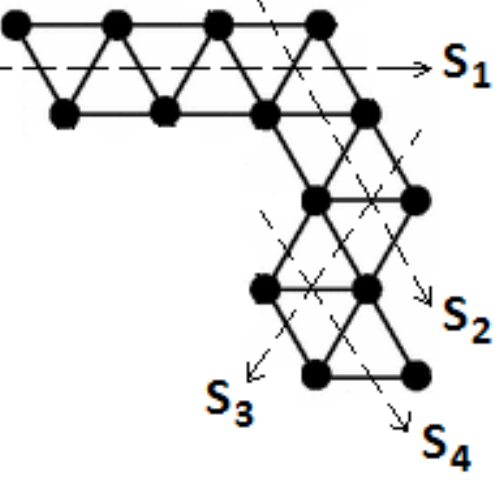}
    \caption{All the segments of a triangular chain.}
    \label{f0}
\end{figure}
Let $E(S_{i})$ (where $1\leq i\leq s$) be the set of all edges of the segment $S_{i}$. Denote by $\mathfrak{T}_{n}$ the collection of all those triangular chains with $n\geq4$ triangles in which every vertex has degree at most five. Note that the graph given in Figure \ref{f1} is not a triangular animal and hence not a triangular chain because one of its interior regions is not a triangle. Now, we can state the main problems of the current study in a more precise form:

\begin{pb}
If $T_{n}\in\mathfrak{T}_{n}$ is a triangular chain with length vector $(l_{1},l_{2},...,l_{s})$, then determine $TI({T}_{n})$.
\end{pb}

\begin{pb}
Among all the triangular chains in the collection $\mathfrak{T}_{n}$, find those for which $TI$ attains its maximum and minimum value.
\end{pb}

The main motivation for considering the triangular chains in the current study, comes from the fact that for every triangular chain $T_{n}$ there exist a benzenoid system whose characteristic graph is isomorphic to $T_{n}$. Hence, the collection $\mathfrak{T}_{n}$ is actually a subclass of the class of all characteristic graphs of benzenoid system.

\section{\textbf{General Expression for BID Indices of Triangular Chains}}

Let $T_{n}\in\mathfrak{T}_{n}$ and suppose that $T_{n}$ has $s$ segments $S_{1}, S_{2},S_{3},...,S_{s}$. In order to obtain the main result of this section, we need to define some structural parameters as follows:
\begin{figure}
    \centering
    \includegraphics[width=1.35in, height=1.2in]{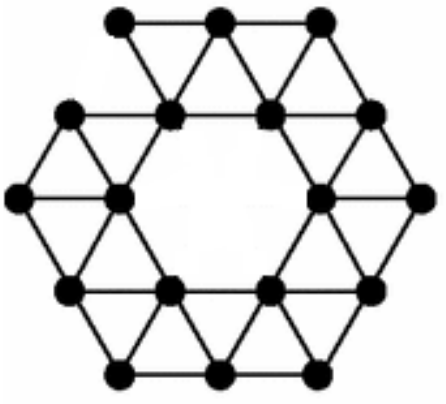}
    \caption{A graph which is not a triangular chain graph.}
    \label{f1}
\end{figure}

\begin{definition}
For $1\leq i\leq s$,
$$\eta_{i}=\eta(S_{i})=
\begin{cases}
1 & \text{if $l_{i}=3$,} \\
0 & \text{otherwise.}
\end{cases}$$
$$\xi_{i}=\xi(S_{i})=
\begin{cases}
1 & \text{if $l_{i}=4$,} \\
0 & \text{otherwise.}
\end{cases}$$
$$\sigma_{i}=\sigma(S_{i})=
\begin{cases}
1 & \text{if $l_{i}=5$,}\\
0 & \text{otherwise.}
\end{cases}$$
\end{definition}

Observe that $T_{n}$ does not contain any nonterminal segment with length three and therefore, if $s\geq3$ then $\eta_{i}=0$ for $2\leq i\leq s-1$.
An edge connecting the vertices of degrees $j$ and $k$ is called edge of the type $(j,k)$. Now, we are in position to establish the general expression for calculating the BID indices of $T_{n}$.
\begin{thm}\label{t1}
Let $T_{n}\in\mathfrak{T}_{n}$ be a triangular chain with length vector $(l_{1},l_{2},...,l_{s})$. Then
\[TI(T_{n})=
\begin{cases}
\Lambda_{0}+\Lambda_{3} & \text{if $s=1,$} \\
\Lambda_{0}+\Lambda_{1}(\eta_{1}+\eta_{2})+\ \Lambda_{2}(\xi_{1}+\xi_{2})+2\Lambda_{3} & \text{if $s=2$,}\\
\Lambda_{0}+\Lambda_{1}(\eta_{1}+\eta_{s})+\Lambda_{2}(\xi_{1}+\xi_{s})
+s\Lambda_{3}+\Lambda_{4}\displaystyle\sum_{i=2}^{s-1}\xi_{i}+\Lambda_{5}\displaystyle\sum_{i=2}^{s-1}\sigma_{i} & \text{if $s\geq3$,}
\end{cases}\]
where
\[\Lambda_{0}=2n\theta_{4,4}+2\theta_{2,3}+2\theta_{2,4}+2\theta_{3,4}-\theta_{3,5}-4\theta_{4,5} \ ,\]
\[\Lambda_{1}=\theta_{2,5}-\theta_{2,4}+\theta_{3,3}-3\theta_{3,4}+\theta_{3,5}+3\theta_{4,4}-2\theta_{4,5} \ ,\]
\[\Lambda_{2}=\theta_{3,5}-\theta_{3,4}+\theta_{4,4}-\theta_{4,5} \ , \
\Lambda_{3}=2\theta_{3,4}+\theta_{3,5}-7\theta_{4,4}+4\theta_{4,5} \ ,\]
\[\Lambda_{4}=2\theta_{3,5}-2\theta_{3,4}+3\theta_{4,4}-4\theta_{4,5}+\theta_{5,5} \ , \
\Lambda_{5}=\theta_{4,4}-2\theta_{4,5}+\theta_{5,5} \ .\]
\end{thm}

\begin{proof}
The result can be easily verified for $s=1,2$. Let us assume that $s\geq3$. Since the triangular chain $T_{n}$ does not contain any vertex with degree greater than or equal to 6, from Equation (\ref{z}) it follows that
\begin{equation}\label{y}
TI(T_{n})=\displaystyle\sum_{2\leq j\leq k\leq 5}x_{j,k}(T_{n})\theta_{j,k} \ .
\end{equation}

\begin{figure}
    \centering
    \includegraphics[width=4.6in, height=1.15in]{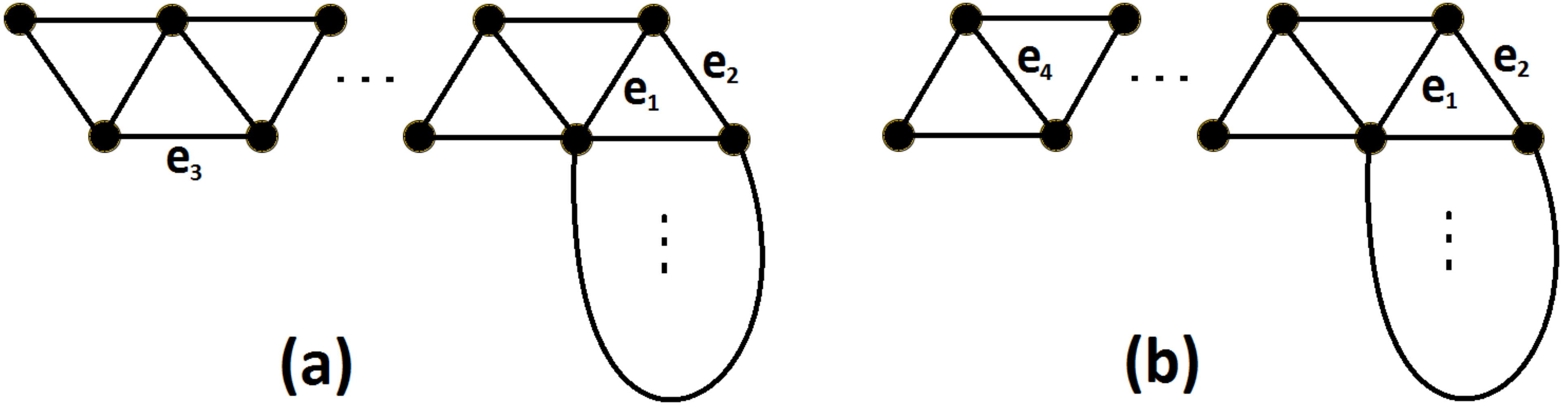}
    \caption{The edges in the first segment, which may be of the type (3,5) are labeled as $e_{1},e_{2},e_{3},e_{4}$. }
    \label{f2}
\end{figure}

\begin{figure}
    \centering
    \includegraphics[width=2.95in, height=1.2in]{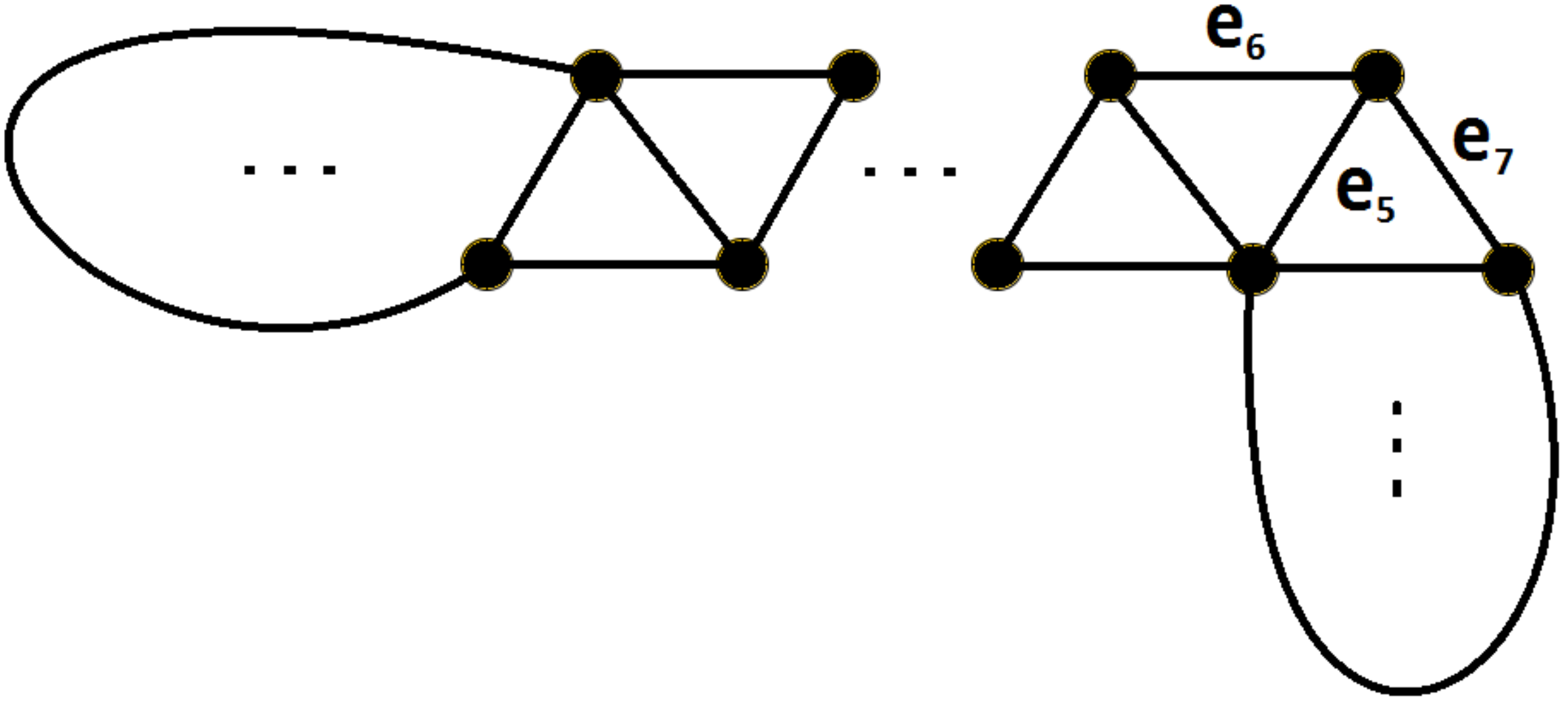}
    \caption{The edges in the $i$th segment (where $2\leq i\leq s-2$ and $s\geq4$), which may be of the type (3,5) are labeled as $e_{5},e_{6},e_{7}$.}
    \label{f3}
\end{figure}
For $2\leq i\leq s$, denote by $x_{j,k}(S_{1})$ and $\widetilde{x}_{j,k}(S_{i})$ the number of edges of type $(j,k)$ belonging to $E(S_{1})$ and $E(S_{i})\setminus E(S_{i-1})$ respectively. Then $x_{j,k}(T_{n})=x_{j,k}(S_{1})+\sum_{i=2}^{s}\widetilde{x}_{j,k}(S_{i})$. To obtain the desired formula, we have to determine $x_{j,k}(T_{n})$ for $2\leq j\leq k\leq5$. Let us start by counting the edges of type $(3,5)$ in $T_{n}$.
Note that the first segment $S_{1}$ contains atleast one edge of the type $(3,5)$, namely $e_{1}$ (see Figure \ref{f2}). Moreover, $e_{2},e_{3}$ or $e_{4}$ (see Figure \ref{f2}) is the edge of type $(3,5)$ if $l_{2}=4,l_{1}=4$ or $l_{1}=3$ respectively. Hence $x_{3,5}(S_{1})=1+\xi_{1}+\xi_{2}+\eta_{1}$. From Figure \ref{f3}, it can be easily seen that the set $E(S_{i})\setminus E(S_{i-1})$ (where $2\leq i\leq s-2$ and $s\geq4$) contains atleast one edge of the type $(3,5)$ namely $e_{5}$ and the edges $e_{6}$, $e_{7}$ are of the type $(3,5)$ if $l_{i}=4$, $l_{i+1}=4$ respectively. This implies that $\widetilde{x}_{3,5}(S_{i})=1+\xi_{i}+\xi_{i+1}$. By analogous reasoning, one have $\widetilde{x}_{3,5}(S_{s-1})=1+\xi_{s-1}+\eta_{s}$ and $\widetilde{x}_{3,5}(S_{s})=\xi_{s}$. Therefore,
if $s=3$, then $x_{3,5}(T_{n})=x_{3,5}(S_{1})+\widetilde{x}_{3,5}(S_{s-1})+\widetilde{x}_{3,5}(S_{s}),$
and if $s\geq4,$ then $x_{3,5}(T_{n})=x_{3,5}(S_{1})+\sum_{i=2}^{s-2}\widetilde{x}_{j,k}(S_{i})
+\widetilde{x}_{3,5}(S_{s-1})+\widetilde{x}_{3,5}(S_{s}).$
In both cases,
\begin{equation}\label{Eq2}
x_{3,5}(T_{n})=s-1+\eta_{1}+\eta_{s}-\xi_{1}-\xi_{s}+2\displaystyle\sum_{i=1}^{s}\xi_{i} \ .
\end{equation}
By simple reasoning and routine calculations, one have
\begin{equation}\label{Eq3}
x_{5,5}(T_{n})=\sum_{i=2}^{s-1}(\xi_{i}+\sigma_{i}),
\end{equation}
\[x_{2,3}(T_{n})=2 \ , \ x_{2,4}(T_{n})=2-\eta_{1}-\eta_{s} \ , \ x_{2,5}(T_{n})=\eta_{1}+\eta_{s} \ , \ n_{2}(T_{n})=2 \ \ \text{and} \]
\begin{equation}\label{Eq4}
\boxed{x_{3,3}(T_{n})=\eta_{1}+\eta_{s} \ , \ n_{3}(T_{n})=s+1 \ , \ n_{4}(T_{n})=n-2s \ , \ n_{5}(T_{n})=s-1.}
\end{equation}
Now, let us consider the following system of equations
\begin{equation}\label{Eq5}
\boxed{\sum_{ \substack{ 2\leq k\leq 5, \\
         k\neq j}}x_{j,k}(T_{n})+2x_{j,j}(T_{n})=j\times n_{j}(T_{n}); \text{ \ \ \ $j=3,4,5$.}}
\end{equation}
Bearing in mind the Equations (\ref{Eq2})-(\ref{Eq4}), we solve the system (\ref{Eq5}) for the unknowns $x_{3,4}(T_{n}),x_{4,4}(T_{n}),x_{4,5}(T_{n})$ and we get
\[x_{3,4}(T_{n})=2s+2-3\eta_{1}-3\eta_{s}+\xi_{1}+\xi_{s}-2\displaystyle\sum_{i=1}^{s}\xi_{i}.\]
\[x_{4,4}(T_{n})=2n-7s+3\eta_{1}+3\eta_{s}+\xi_{1}+\xi_{s}+3\displaystyle\sum_{i=2}^{s-1}\xi_{i}
+\displaystyle\sum_{i=2}^{s-1}\sigma_{i}.\]
\[x_{4,5}(T_{n})=4s-4-2\eta_{1}-2\eta_{s}-\xi_{1}-\xi_{s}-4\displaystyle\sum_{i=2}^{s-1}\xi_{i}
-2\displaystyle\sum_{i=2}^{s-1}\sigma_{i}.\]
By substituting the values of $x_{j,k}(T_{n})$ (where $2\leq j\leq k\leq5$) in Equation (\ref{y}), one arrives at the desired formula.
\end{proof}

\section{\textbf{Extremal Triangular Chains for BID Indices}}

To characterize the extremal triangular chains in $\mathfrak{T}_{n}$ with respect to BID indices, let us defined the structural parameter $\Phi_{TI}$, for any $T_{n}\in\mathfrak{T}_{n}$, as follows:
\[\Phi_{TI}(T_{n})=
\begin{cases}
\Lambda_{3} & \text{if $s=1$} \\
\Lambda_{1}(\eta_{1}+\eta_{2})+\Lambda_{2}(\xi_{1}+\xi_{2})+2\Lambda_{3} & \text{if $s=2$,}
\end{cases}\]
and for $s\geq3$,
\[\Phi_{TI}(T_{n})=\sum_{i=1}^{s}\Phi_{TI}(S_{i})=\Lambda_{1}(\eta_{1}+\eta_{s})+\Lambda_{2}(\xi_{1}+\xi_{s})
+s\Lambda_{3}+\Lambda_{4}\sum_{i=2}^{s-1}\xi_{i}+\Lambda_{5}\sum_{i=2}^{s-1}\sigma_{i} \ ,\]
where
\[\Phi_{TI}(S_{1})=\Lambda_{1}\eta_{1}+\Lambda_{2}\xi_{1}+\Lambda_{3} \ ,
\Phi_{TI}(S_{s})=\Lambda_{1}\eta_{s}+\Lambda_{2}\xi_{s}+\Lambda_{3} \ ,\]
\[\Phi_{TI}(S_{i})=\Lambda_{3}+\Lambda_{4}\xi_{i}+\Lambda_{5}\sigma_{i} \ ; \ \ \ 2\leq i\leq s-1.\]
Bearing in mind the definition of $\Phi_{TI}$ and the Theorem \ref{t1}, one have
\begin{cor}\label{cr0}
Among all the triangular chains in the collection $\mathfrak{T}_{n}$ , a triangular chain has the maximum (respectively minimum) $TI$ value if and only if it has the maximum (respectively minimum) $\Phi_{TI}$ value.
\end{cor}


Denote by $L_{n}$ the linear triangular chain with $n\geq4$ triangles. By a zigzag triangular chain $Z_{n}$, we mean a triangular chain with $n\geq4$ triangles and length vector $(a,\underbrace{4,4,4,...,4}_{\left(\lfloor\frac{n}{2}\rfloor-2\right)-\text{times}},b)$ where $a,b\leq4$ and at least one of $a,b$ is 3.

\begin{cor}\label{cr1}
Suppose that $\Lambda_{1},\Lambda_{2},...,\Lambda_{5}$ are the quantities defined in Theorem \ref{t1} and let $T_{n}\in\mathfrak{T}_{n}$. \\
\textit{1.} If $\Lambda_{i}<0$ for $i=1,2,3,4$ and $-\Lambda_{3}>\Lambda_{5}>0$, then $TI(T_{n})\leq TI(L_{n})$ with equality if and only if $T_{n}\cong L_{n}$.\\
\textit{2.} If $\Lambda_{i}>0$ for $i=1,2,3,4$ and $-\Lambda_{3}<\Lambda_{5}<0$, then $TI(T_{n})\geq TI(L_{n})$ with equality if and only if $T_{n}\cong L_{n}$.\\
\end{cor}
\begin{proof}
\textit{1.}
From the definition of $\Phi_{TI}$ , it follows that $\Phi_{TI}(L_{n})=\Lambda_{3}$. For $s=2$, one have
\[\Phi_{TI}(T_{n})=\Phi_{TI}(S_{1})+\Phi_{TI}(S_{2})=\Lambda_{1}(\eta_{1}+\eta_{2})+\Lambda_{2}(\xi_{1}+\xi_{2})
+2\Lambda_{3}\leq2\Lambda_{3}<\Lambda_{3}.\]
If $s\geq3$ then for $2\leq i\leq s-1$, one have
\[\Phi_{TI}(S_{i})=\Lambda_{3}+\Lambda_{4}\xi_{i}+\Lambda_{5}\sigma_{i}\leq\Lambda_{3}+\Lambda_{5}<0,\]
and hence $\Phi_{TI}(T_{n})=\sum_{i=1}^{s}\Phi_{TI}(S_{i})<\Lambda_{3}$. Therefore $\Phi_{TI}(T_{n})\leq\Lambda_{3}$ with equality if and only if $T_{n}\cong L_{n}$. From Corollary \ref{cr0}, desired result follows.\\

\textit{2.} The proof of second part is fully analogous to that of first part.
\end{proof}

\begin{cor}\label{cr2}
Suppose that $\Lambda_{1},\Lambda_{2},...,\Lambda_{5}$ are the quantities defined in Theorem \ref{t1} and let $T_{n}\in\mathfrak{T}_{n}$. \\
\textit{1.} If $-\Lambda_{3}>\Lambda_{5}$, $\Lambda_{i}$ is negative for $i=1,2,3,4$, $2\Lambda_{4}<\Lambda_{1}<\Lambda_{2}$ and $\Lambda_{1}+\Lambda_{5}>\Lambda_{2}+\Lambda_{4}$, then $TI(T_{n})\geq TI(Z_{n})$ with equality if and only if $T_{n}\cong Z_{n}$.\\
\textit{2.} If $-\Lambda_{3}<\Lambda_{5}$, $\Lambda_{i}$ is positive for $i=1,2,3,4$, $2\Lambda_{4}>\Lambda_{1}>\Lambda_{2}$ and $\Lambda_{1}+\Lambda_{5}<\Lambda_{2}+\Lambda_{4}$, then $TI(T_{n})\leq TI(Z_{n})$ with equality if and only if $T_{n}\cong Z_{n}$.
\end{cor}

\begin{proof}
\textit{1.} The result can be easily justified for $n\leq6$, so let us assume that $n\geq7$.  Let $\widetilde{T}_{n}\in\mathfrak{T}_{n}$ such that $\Phi_{TI}(\widetilde{T}_{n})$ is minimum. Note that $\Phi_{TI}(L_{n})>\Phi_{TI}(Z_{n})$ and hence $\widetilde{T}_{n}\ncong L_{n}$. Suppose that $\widetilde{T}_{n}$ has length vector $(l_{1},l_{2},...,l_{s})$. If at least one of $l_{1},l_{s}$ is greater than 4, say $l_{1}\geq5$. Then the triangular chain $T^{(0)}_{n}$ with length vector $(3,l_{1}-1,l_{2},l_{3},...,l_{s})$ belongs to $\mathfrak{T}_{n}$ and
\[\Phi_{TI}(T^{(0)}_{n})-\Phi_{TI}(\widetilde{T}_{n})=
\begin{cases}
\Lambda_{1}+\Lambda_{3}+\Lambda_{4} & \text{if $l_{1}=5$,} \\
\Lambda_{1}+\Lambda_{3}+\Lambda_{5} & \text{if $l_{1}=6$,} \\
\Lambda_{1}+\Lambda_{3} & \text{if $l_{1}\geq7$.}
\end{cases}\]
It is easy to see that $\Phi_{TI}(T^{(0)}_{n})-\Phi_{TI}(\widetilde{T}_{n})<0$. This contradicts the minimality of $\Phi_{TI}(\widetilde{T}_{n})$. Therefore $l_{1},l_{s}\leq4$, which implies that $s\geq3$ (since $n\geq7$). Now, we will prove that at least one of $l_{1},l_{s}$ is 3. Suppose to the contrary that $l_{1}=l_{s}=4$. Note that the triangular chain $T^{(00)}_{n}$ with length vector $(3,l_{1},l_{2},...,l_{s-1},l_{s}-1)$ is a member of $\mathfrak{T}_{n}$ and
\[\Phi_{TI}(T^{(00)}_{n})-\Phi_{TI}(\widetilde{T}_{n})=
2\Lambda_{1}-2\Lambda_{2}+\Lambda_{3}+\Lambda_{4}<0,\]
which is a contradiction. Hence, at least one of $l_{1},l_{s}$ is 3. Without loss of generality we can assume that $l_{1}=3$. Now, if $l_{j}\geq5$ for some $j$ where $2\leq j\leq s-1$. We consider two cases.

\textit{Case 1.} If $l_{j}=5$, then the triangular chain $T^{(000)}_{n}$ with length vector $(l_{1}+1,l_{2},l_{3},...,l_{j-1},l_{j}-1,l_{j+1},...,l_{s})$ belongs to $\mathfrak{T}_{n}$ and
\[\Phi_{TI}(T^{(000)}_{n})-\Phi_{TI}(\widetilde{T}_{n})=
-\Lambda_{1}+\Lambda_{2}+\Lambda_{4}-\Lambda_{5}<0,\]
a contradiction.

\textit{Case 2.} If $l_{j}\geq6$, then the triangular chain $T^{(0000)}_{n}$ with length vector $(l_{j}-3,l_{1}+1,l_{2},l_{3},...,l_{j-1},4,l_{j+1},...,l_{s})$ is an element of $\mathfrak{T}_{n}$ and
\[\Phi_{TI}(T^{(0000)}_{n})-\Phi_{TI}(\widetilde{T}_{n})=
(x-1)\Lambda_{1}+y\Lambda_{2}+\Lambda_{3}+2\Lambda_{4},\]
where $x,y\in\{0,1\}$ such that at least one of $x,y$ is zero. Note that $\Phi_{TI}(T^{(0000)}_{n})-\Phi_{TI}(\widetilde{T}_{n})$ is negative and hence a contradiction is obtained.\\
Hence we conclude that $\widetilde{T}_{n}\cong Z_{n}$. By using the Corollary \ref{cr0}, we arrive at the desired result.\\

\textit{2.} The proof of second part is completely analogous to that of part 1.
\end{proof}

The choices $\theta_{a,b}=\frac{1}{\sqrt{ab}},\frac{2\sqrt{ab}}{a+b},\frac{1}{\sqrt{a+b}},\frac{1}{ab},ln(a+b),\frac{2}{a+b}$ (where $ln$ denotes the natural logarithm) in Equation (1.1) correspond to the Randi\'{c} index, first geometric-arithmetic index, sum-connectivity index, modified second Zagreb index, natural logarithm of the multiplicative sum Zagreb index, harmonic index respectively.

\begin{cor}
Let $T_{n}$ be any triangular chain in the collection $\mathfrak{T}_{n}$. \\
\textit{1.} If $TI$ is one of the following topological indices: sum-connectivity index, Randi\'{c} index, harmonic index, first geometric-arithmetic index and modified second Zagreb index. Then
\[TI(Z_{n})\leq TI(T_{n})\leq TI(L_{n}),\]
with left (respectively right) equality if and only if $T_{n}\cong Z_{n}$ (respectively $T_{n}\cong L_{n}$);\\
\textit{2.} For the multiplicative sum Zagreb index, the following inequality holds:
\[\Pi_{1}^{*}(L_{n}) \leq \Pi_{1}^{*}(T_{n})\leq \Pi_{1}^{*}(Z_{n}),\]
with left (respectively right) equality if and only if $T_{n}\cong L_{n}$ (respectively $T_{n}\cong Z_{n}$).
\end{cor}

\begin{proof}
\textit{1.} By routine computations one can easily verified that the hypothesis of Corollary \ref{cr1}(1) and Corollary \ref{cr2}(1) are satisfied for each of the following indices: sum-connectivity index, Randi\'{c} index, harmonic index, first geometric-arithmetic index and modified second Zagreb index. Hence the desired result follows from Corollary \ref{cr1}(1) and Corollary \ref{cr2}(1).\\

\textit{2.} It can be easily checked that the hypothesis of Corollary \ref{cr1}(2) and Corollary \ref{cr2}(2) are satisfied for $ln[\Pi_{1}^{*}]$ and hence one have
\begin{equation}\label{Eq11111}
ln[\Pi_{1}^{*}(L_{n})] \leq ln[\Pi_{1}^{*}(T_{n})]\leq ln[\Pi_{1}^{*}(Z_{n})],
\end{equation}
with left (respectively right) equality if and only if $T_{n}\cong L_{n}$ (respectively $T_{n}\cong Z_{n}$). Since the exponential function is strictly increasing and this function is inverse of the natural logarithm function. Therefore, from the Inequality \ref{Eq11111} the required result follows.
\end{proof}

The Equation (\ref{z}) gives the augmented Zagreb index ($AZI$) if we take $\theta_{a,b}=\left(\frac{ab}{a+b-2}\right)^{3}$. For $n\geq6$, denote by $T^{-}_{n}$ the triangular chain with the length vector $(3,x,3)$ where $x\geq4$. Note that $T^{-}_{6}\cong Z_{6}$.

\begin{cor}
Let $T_{n}$ be any triangular chain in the collection $\mathfrak{T}_{n}$. Then
\[AZI(T_{n})\geq
\begin{cases}
AZI(Z_{n})  & \text{if $n\leq8$,} \\
AZI(T^{-}_{n}) & \text{otherwise.}
\end{cases}\]
The equality in the first case holds if and only if $T_{n}\cong Z_{n}$ and the equality in the second case holds if and only if $T_{n}\cong T^{-}_{n}$.

\end{cor}

\begin{proof}
By routine computations, one have
\[\Lambda_{1}\approx-4.2147, \Lambda_{2}\approx-2.5597, \Lambda_{3}\approx3.8267, \Lambda_{4}\approx-2.2860, \Lambda_{5}\approx2.8333.\]
The result can be easily verified for $n\leq10$. So, we assume that $n\geq11$ and $T_{n}\ncong T^{-}_{n}$. After simple calculations one have $\Phi_{AZI}(T^{-}_{n})\approx3.0507$ and hence $\Phi_{AZI}(L_{n})>\Phi_{AZI}(T^{-}_{n})$. We discuss four cases.

\textit{Case 1.} If $s=2$, then at least one of $l_{1},l_{2}$ must be greater than 4, which implies that
\[\Phi_{AZI}(T_{n})\approx 7.6534-4.2147(\eta_{1}+\eta_{2})-2.5597(\xi_{1}+\xi_{2})\geq3.4387>\Phi_{AZI}(T^{-}_{n}).\]

\textit{Case 2.} If $s=3$, then the inequality $n\geq11$ implies that $l_{i}\geq5$ for at least one $i$ (where $i=1,2,3$). Here we consider two subcases:

\textit{Subcase 2.1.} If at least one of $l_{1},l_{3}$ is greater than 4, then
\[\Phi_{AZI}(S_{1})+\Phi_{AZI}(S_{3})\approx 7.6534-4.2147(\eta_{1}+\eta_{3})-2.5597(\xi_{1}+\xi_{3})\geq3.4387,\]
and
\[\Phi_{AZI}(S_{2})\approx 3.8267-2.286\xi_{2}+2.8333\sigma_{2}\geq1.5407.\]
This leads to
\[\Phi_{AZI}(T_{n})=\sum_{i=1}^{3}\Phi_{AZI}(S_{i})> \Phi_{AZI}(T^{-}_{n}).\]

\textit{Subcase 2.2.} If $l_{2}\geq5$, then at least one of $l_{1},l_{3}$ must be greater than 3 (since $T_{n}\ncong T^{-}_{n}$) which implies that
\[\Phi_{AZI}(S_{1})+\Phi_{AZI}(S_{3})\approx 7.6534-4.2147(\eta_{1}+\eta_{3})-2.5597(\xi_{1}+\xi_{3})\geq0.879,\]
and
\[\Phi_{AZI}(S_{2})\approx 3.8267-2.286\xi_{2}+2.8333\sigma_{2}\geq3.8267.\]
Hence, it follows that
\[\Phi_{AZI}(T_{n})=\sum_{i=1}^{3}\Phi_{AZI}(S_{i})> \Phi_{AZI}(T^{-}_{n}).\]

\textit{Case 3.} If $s=4$, then the inequality $n\geq11$ implies that $l_{i}\geq5$ for at least one $i$ (where $1\leq i\leq4$). We have two possibilities:

\textit{Subcase 3.1.} If at least one of $l_{1},l_{4}$ is greater than 4, then
\[\Phi_{AZI}(S_{1})+\Phi_{AZI}(S_{4})\approx 7.6534-4.2147(\eta_{1}+\eta_{4})-2.5597(\xi_{1}+\xi_{4})\geq3.4387,\]
and
\[\Phi_{AZI}(S_{2})+\Phi_{AZI}(S_{3})\approx 7.6534-2.286(\xi_{2}+\xi_{3})+2.8333(\sigma_{2}+\sigma_{3})\geq3.0814.\]
Hence
\[\Phi_{AZI}(T_{n})=\sum_{i=1}^{4}\Phi_{AZI}(S_{i})> \Phi_{AZI}(T^{-}_{n}).\]

\textit{Subcase 3.2.} If at least one of $l_{2},l_{3}$ is greater than 4, then
\[\Phi_{AZI}(S_{2})+\Phi_{AZI}(S_{3})\approx 7.6534-2.286(\xi_{2}+\xi_{3})+2.8333(\sigma_{2}+\sigma_{3})\geq5.3674,\]
and
\[\Phi_{AZI}(S_{1})+\Phi_{AZI}(S_{4})\approx 7.6534-4.2147(\eta_{1}+\eta_{4})-2.5597(\xi_{1}+\xi_{4})\geq-0.776.\]
Hence
\[\Phi_{AZI}(T_{n})=\sum_{i=1}^{4}\Phi_{AZI}(S_{i})> \Phi_{AZI}(T^{-}_{n}).\]

\textit{Case 4.} If $s\geq5$, then for $2\leq i\leq s-1$,
\[\Phi_{AZI}(S_{i})\approx 3.8267-2.286\xi_{i}+2.8333\sigma_{i}\geq1.5407,\]
and
\[\Phi_{AZI}(S_{1})+\Phi_{AZI}(S_{s})\approx 7.6534-4.2147(\eta_{1}+\eta_{s})-2.5597(\xi_{1}+\xi_{s})\geq-0.776.\]
Bearing in mind the fact $s\geq5$, one have
\[\Phi_{AZI}(T_{n})=\sum_{i=1}^{s}\Phi_{AZI}(S_{i})>\Phi_{AZI}(T^{-}_{n}).\]
In all cases, we arrive at $\Phi_{AZI}(T_{n})>\Phi_{AZI}(T^{-}_{n})$. Therefore, from Corollary \ref{cr0} the required result follows.
\end{proof}

The choice $\theta_{a,b}=\mid a-b\mid$ (respectively $\theta_{a,b}=ab$) in Equation (\ref{z}) gives the Albertson index $A$ (respectively second Zagreb index $M_{2}$). Let $\mathfrak{T}^{*}_{n}$ be the subclass of $\mathfrak{T}_{n}$ of all those triangular chains with $n\geq7$ triangles (where $n$ is odd) in which both terminal segments have length 3, exactly one nonterminal segment has length 5 and all the other nonterminal segments (if exist) have length 4.

\begin{cor}
Let $T_{n}$ be any triangular chain in the collection $\mathfrak{T}_{n}$.\\
\textit{1.} For the Albertson index $A$, the following inequality holds
\[10\leq A(T_{n})\leq
\begin{cases}
3n+2 & \text{if $n$ is even}, \\
3n+1 & \text{otherwise}.
\end{cases} \]
The upper bounds are attained if and only if $T_{n}\cong Z_{n}$ and the lower bound is attained if and only if $T_{n}\cong L_{n}$.\\
\textit{2.} For the second Zagreb index $M_{2}$, the following inequality holds
\[4(8n-9)\leq M_{2}(T_{n})\leq
\begin{cases}
128 & \text{if $n=5$}, \\
35n-45 & \text{if $n$ is even},\\
35n-46 & \text{otherwise}.
\end{cases}\]
The lower bound is attained if and only if $T_{n}\cong L_{n}$, the first and second upper bounds are attained if and only if $T_{n}\cong Z_{n}$, and the third upper bound is attained if and only if $T_{n}\in\mathfrak{T}^{*}_{n}$.

\end{cor}

\begin{proof}\textit{1.}
For the Albertson index $A$, one have $-\Lambda_{0}=\Lambda_{1}=\Lambda_{4}=\Lambda_{5}=-2$, $\Lambda_{2}=0$ and $\Lambda_{3}=8=\Phi_{A}(L_{n})$. Firstly, we establish the lower bound. Let $T_{n}\ncong L_{n}$, then for $s=2$ one have $\Phi_{A}(T_{n})=16-2(\eta_{1}+\eta_{2})\geq12>\Phi_{A}(L_{n})$. If $s\geq3$ then it follows that $\Phi_{A}(S_{1})=8-2\eta_{1}\geq6$ , $\Phi_{A}(S_{s})=8-2\eta_{s}\geq6$, $\Phi_{A}(S_{i})=8-2\xi_{i}-2\sigma_{i}\geq6$ where $2\leq i\leq s-1$ and therefore
\[\Phi_{A}(T_{n})=\sum_{i=1}^{s}\Phi_{AZI}(S_{i})\geq6s>\Phi_{A}(L_{n}).\]
Hence from Corollary \ref{cr0}, we have $A(T_{n})\geq A(L_{n})$ with equality if and only if $T_{n}\cong L_{n}$.

To establish the upper bound, let us choose $\widetilde{T}_{n}\in\mathfrak{T}_{n}$ such that $\Phi_{A}(\widetilde{T}_{n})$ is maximum. It can be easily checked that $\Phi_{A}(Z_{n})>\Phi_{A}(L_{n})$, which implies that $\widetilde{T}_{n}\not\cong L_{n}$. Suppose that $\widetilde{T}_{n}$ has length vector $(l_{1},l_{2},...,l_{s})$ where $s\geq2$.

\textit{Claim 1.} $l_{1},l_{s}\leq4$ and at least one of $l_{1},l_{s}$ is 3.\\
If at least one of $l_{1},l_{s}$ is greater than or equal to 5. Without loss of generality, suppose that $l_{1}\geq5$. Then the triangular chain ${T}^{(1)}_{n}$
with length vector $(3,l_{1}-1,l_{2},l_{3},...,l_{s})$ belongs to $\mathfrak{T}_{n}$ and $\Phi_{A}(\widetilde{T}_{n})-\Phi_{A}(T^{(1)}_{n})\leq-4,$ which is a contradiction to the maximality of $\Phi_{A}(\widetilde{T}_{n})$.
If $l_{1}=l_{s}=4$, then the triangular chain ${T}^{(2)}_{n}$ with length vector $(3,l_{1},l_{2},...,l_{s-1},3)$ is a member of $\mathfrak{T}_{n}$ and $\Phi_{A}(\widetilde{T}_{n})-\Phi_{A}(T^{(2)}_{n})=-2,$ which is again a contradiction.

\textit{Claim 2.} If one of $l_{1},l_{s}$ is 3 and the other is 4, then every nonterminal segment (if exists) has length 4.\\
Without loss of generality, we assume that $l_{1}=4$ and $l_{s}=3$. Suppose to the contrary that at least one nonterminal segment, say (without loss of generality) $l_{2}$ has length greater than or equal to 5. Then the triangular chain ${T}^{(3)}_{n}$ with length vector $(3,4,l_{2}-1,l_{3},l_{4},...,l_{s})$ is an element of $\mathfrak{T}_{n}$ and $\Phi_{A}(\widetilde{T}_{n})<\Phi_{A}(T^{(3)}_{n}),$ a contradiction.

\textit{Claim 3.} Every nonterminal segment (if exists) has length less than or equal to 5 and at most one nonterminal segment (if exists) has length 5.\\
If $\widetilde{T}_{n}$ has at least one nonterminal segment of length greater than or equal to 6. Without loss of generality, one can assume that $l_{2}\geq6$. Then the triangular chain ${T}^{(4)}_{n}$ with length vector $(l_{1},4,l_{2}-2,l_{3},l_{4},...,l_{s})$ belongs to $\mathfrak{T}_{n}$ and $\Phi_{A}(\widetilde{T}_{n})<\Phi_{A}(T^{(4)}_{n}),$ this contradicts the maximality of $\Phi_{A}(\widetilde{T}_{n})$. If there exist at least two nonterminal segments with length 5. Without loss of generality, suppose that $l_{2}=l_{3}=5$, then the triangular chain ${T}^{(5)}_{n}$ with length vector $(l_{1},4,4,4,l_{4},l_{5},...,l_{s})$ is a member of $\mathfrak{T}_{n}$ and $\Phi_{A}(\widetilde{T}_{n})<\Phi_{A}(T^{(5)}_{n})$, again a contradiction.

From Claim 1, Claim 2 and Claim 3 it follows that either $\widetilde{T}_{n}\cong Z_{n}$ or $\widetilde{T}_{n}\cong T^{*}_{n}\in\mathfrak{T}^{*}_{n}$. If $\widetilde{T}_{n}\cong T^{*}_{n}\in\mathfrak{T}^{*}_{n}$ then $n\geq7$ and $n$ is odd. But
\[\Phi_{A}(Z_{n})=3n-1>\Phi_{A}(T^{*}_{n})=3(n-1),\]
a contradiction. Therefore $\widetilde{T}_{n}\cong Z_{n}$. After simple calculations, one have
\[A(Z_{n})=
\begin{cases}
3n+2 & \text{if $n$ is even}, \\
3n+1 & \text{otherwise}.
\end{cases} \]
From Corollary \ref{cr0}, the desired result follows.\\

\textit{2.} For the second Zagreb index $M_{2}$, one have
\[\Lambda_{0}=32n-43, \ \Lambda_{1}=-2, \ \Lambda_{2}=\Lambda_{4}=-1, \ \Lambda_{3}=7, \ \Lambda_{5}=1.\]
Also, note that if $n\geq7$ and $n$ is odd then
\[\Phi_{M_{2}}(Z_{n})=3n-4 \ , \ \Phi_{M_{2}}(T^{*}_{n})=3(n-1).\]
Now, using the same technique that was used to prove the first part of the theorem, we arrive at the desired result.
\end{proof}

The choice $\theta_{a,b}= \sqrt{\frac{a+b-2}{ab}}$ in Equation (1.1) corresponds to the atom-bond connectivity ($ABC$) index. For the $ABC$ index, it can be easily verified that $-\Lambda_{1}-\Lambda_{3}<\Lambda_{5}$, $\Lambda_{i}$ is positive for $i=1,2,3,4$, $2\Lambda_{4}>\Lambda_{1}>\Lambda_{2}$ and $\Lambda_{1}+\Lambda_{5}<\Lambda_{2}+\Lambda_{4}$. On the other hand, it can be easily observed that if the condition $-\Lambda_{3}<\Lambda_{5}$ in Corollary \ref{cr2}(2) is replaced with $-\Lambda_{1}-\Lambda_{3}<\Lambda_{5}<0$, then the conclusion remains true and hence we have $ABC(T_{n})\leq TI(Z_{n})$ with equality if and only if $T_{n}\cong Z_{n}$.

\section{\textbf{Concluding Remarks}}

For any triangular chain $T_{n}\in\mathfrak{T}_{n}$, we have established a closed form formula given in Theorem \ref{t1} for calculating the BID indices. Then, using this formula, we have characterized the extremal triangular chains with respect to a variety of famous BID indices over the collection $\mathfrak{T}_{n}$. More precisely, among all the triangular chains in $\mathfrak{T}_{n}$, we have characterized the extremal ones for the sum-connectivity index, Randi\'{c} index, harmonic index, first geometric-arithmetic index, second Zagreb index, modified second Zagreb index, multiplicative sum Zagreb index and Albertson index. In addition, we have distinguish the triangular chains in $\mathfrak{T}_{n}$ with the minimum augmented Zagreb index and maximum atom-bond connectivity index. However, the problem of characterizing the triangular chains with maximum augmented Zagreb index and minimum atom-bond connectivity index in the collection $\mathfrak{T}_{n}$ remains open. Moreover, it seems to be interesting to extend the results of current study for the collection of all triangular chains.

\section{\textbf{Acknowledgements}}

The authors are very grateful to Professor Fuji Zhang for helpful discussion on $k$-polygonal chains and Professor Clive Elphick for providing the paper \cite{gutman}. The authors would also like to express their sincere gratitude to the anonymous referee for his/her valuable comments, which led to a number of improvements in the earlier version of the manuscript.


\begin{thebibliography}{00}

\bibitem{a1} M. O. Albertson, The Irregularity of a Graph, \textit{Ars Combinatoria} \textbf{46}, (1997) 219-225.

\bibitem{AA2} A. Ali, A. A. Bhatti, Z. Raza, A note on the zeroth-order general Randi\'{c} index of cacti and polyomino chains, \textit{Iranian J. Math. Chem.} \textbf{5}, (2014) 143-152.

\bibitem{AA1} A. Ali, Z. Raza, A. A. Bhatti, Bond incident degree (BID) indices for some nanostructures, \textit{Optoelectron. Adv. Mat.} \textbf{10},(2016) 108-112.

\bibitem{AA3} A. Ali, Z. Raza, A. A. Bhatti, Bond incident degree (BID) indices of polyomino chains: a unified approach, \textit{Appl. Math. Comp.} \textbf{287}, (2016) 28-37.

\bibitem{AA4} A. Ali, Z. Raza, A. A. Bhatti, Extremal pentagonal chains with respect to bond Incident degree (BID) indices, submitted.

\bibitem{An2} M. An, L. Xiong, Extremal polyomino chains with respect to general Randi\'{c} index, \textit{J. Comb. Optim.} \textbf{31} (2), (2016) 635-647.



\bibitem{d4} H. Deng, J. Yang, F. Xia, A general modeling of some vertex-degree based topological indices in benzenoid systems and phenylenes, \textit{Comput. Math. Appl.} \textbf{61}, (2011) 3017-3023.

\bibitem{deng} H. Deng, S. Balachandran, S. K. Ayyaswamy, Y. B. Venkatakrishnan, The harmonic indices of polyomino chains, \textit{Natl. Acad. Sci. Lett.} \textbf{37}(5), (2014) 451-455.

\bibitem{De99} J. Devillers, A.T. Balaban (Eds.), \textit{Topological Indices and Related Descriptors in QSAR and QSP}R, Gordon and Breach, Amsterdam, (1999).

\bibitem{d2} M.V. Diudea (Ed.),\textit{QSPR/QSAR Studies by Molecular Descriptors}, Nova, Huntington, (2001).

\bibitem{dob} A. A. Dobrynin, I. Gutman, S. Klav\v{z}ar, P. \v{Z}igert, Wiener index of hexagonal systems, \textit{Acta Appl. Math.} \textbf{72}, (2002) 247-294.

\bibitem{e1} M. Eliasi, I. Gutman, A. Iranmanesh, Multiplicative versions of first Zagreb index, \textit{MATCH Commun. Math. Comput. Chem.} \textbf{68}, (2012) 217-230.

\bibitem{e3} E. Estrada, L. Torres, L. Rodr\'{i}guez, I. Gutman, An atom-bond connectivity index: modelling the enthalpy of formation of alkanes, \textit{Indian J. Chem. A} \textbf{37}, (1998) 849-855.

\bibitem{EsBo2013} E. Estrada, D. Bonchev, Section 13.1. Chemical Graph Theory, in \textit{Handbook of Graph Theory}, 2nd ed., Gross, Yellen and Zhang, Eds., CRC Press, Boca Raton, FL, (2013), pp. 1538-1558.

\bibitem{f2} S. Fajtlowicz, On conjectures of Graffiti-II, \textit{Congr. Numer.} \textbf{60}, (1987) 187-197.

\bibitem{f4} B. Furtula, A. Graovac, D. Vuki\v{c}evi\'{c}, Augmented Zagreb index, \textit{J. Math. Chem.} \textbf{48}, (2010) 370-380.

\bibitem{f5} B. Furtula, I. Gutman, M. Dehmer, On structure-sensitivity of degree-based topological indices, \textit{Appl. Math. Comput.} \textbf{219}, (2013) 8973-8978.

\bibitem{g11} S. W. Golomb, Checker boards and polyominoes, \textit{Amer. Math. Monthly} \textbf{61}, (1954) 675-682.

\bibitem{g12} S. W. Golomb, Polyominoes: \textit{Puzzles, Patterns, Problems, and Packings}, 2nd ed., Princeton Uni. Press, Princeton, NJ, (1994).


\bibitem{gutman89} I. Gutman, S. J. Cyvin, \textit{Introduction to the Theory of Benzenoid Hydrocarbons}, Springer Verlag, Berlin, Germany, (1989).

\bibitem{gutman} I. Gutman, B. Furtula, C. Elphick, Three new/old vertex-degree-based topological indices, \textit{MATCH Commun. Math. Comput. Chem.} \textbf{72}, (2014) 617-632.

\bibitem{g22} I. Gutman, B. Ru\v{s}\v{c}i\'{c}, N. Trinajsti\'{c}, C. F. Wilcox, Graph theory and molecular orbitals. XII. Acyclic polyenes, \textit{J. Chem. Phys.} \textbf{62}, (1975) 3399-3405.

\bibitem{g2} I. Gutman, J. To\v{s}ovi\'{c}, Testing the quality of molecular structure descriptors: Vertex-degree-based topological indices, \textit{J. Serb. Chem. Soc.} \textbf{78}, (2013) 805-810.

\bibitem{f1} I. Gutman, B. Furtula (Eds.), \textit{Novel Molecular Structure Descriptors—Theory and Applications vols. I-II}, Univ. Kragujevac, Kragujevac, (2010).

\bibitem{g1} I. Gutman, Degree-based topological indices, \textit{Croat. Chem. Acta} \textbf{86}(4), (2013) 351-361.

\bibitem{Ha69} F. Harary, \textit{Graph Theory}, Addison-Wesley, Reading, MA, (1969).

\bibitem{hollas} B. Hollas, The covariance of topological indices that depend on the degree of a vertex,\textit{ MATCH Commun. Math. Comput. Chem.} \textbf{54}(1), (2005) 177-187.

\bibitem{2} W. Karcher, J. Devillers, \textit{Practical Applications of Quantitative Structure-Activity Relationships (QSAR) in Environmental Chemistry and Toxicology}, Springer, (1990).

\bibitem{n1} S. Nikoli\'{c}, G. Kova\v{c}evi\'{c}, A. Mili\v{c}evi\'{c}, N. Trinajsti\'{c}, The zagreb indices 30 years after, \textit{Croat. Chem. Acta} \textbf{76}, (2003) 113-124.


\bibitem{r1} J. Rada, R. Cruz, I. Gutman, Vertex-degree-based topological indices of catacondensed hexagonal systems, \textit{Chem. Phys. Lett.} \textbf{ 572}, (2013) 154-157.

\bibitem{r3} M. Randi\'{c}, On characterization of molecular branching, \textit{J. Am. Chem. Soc.} \textbf{97}, (1975) 6609-6615.

\bibitem{ToCo2000} R. Todeschini, V. Consonni, \textit{Handbook of Molecular Descriptors}, Wiley-VCH, Weinheim, (2000).

\bibitem{Tr92} N. Trinajsti\'{c}, \textit{Chemical Graph Theory}, 2nd revised ed., CRC Press, Boca Raton, Florida, (1992).

\bibitem{VuDe10} D. Vuki\v{c}evi\'{c}, M. Ga\v{s}perov, Bond additive modeling 1. Adriatic indices, \textit{Croat. Chem. Acta} \textbf{83} (3), (2010) 243-260.

\bibitem{Vu10} D. Vuki\v{c}evi\'{c}, Bond additive modeling 2. Mathematical properties of max-min rodeg index, \textit{Croat. Chem. Acta} \textbf{83} (3), (2010) 261-273.

\bibitem{VuDu11} D. Vuki\v{c}evi\'{c}, J. Durdevi\'{c}, Bond additive modeling 10. Upper and lower bounds of bond incident degree indices of catacondensed fluoranthenes,  \textit{Chem. Phys. Lett.} \textbf{515}, (2011) 186-189.

\bibitem{v1} D. Vuki\v{c}evi\'{c}, B. Furtula, Topological index based on the ratios of geometrical and arithmetical means of end-vertex degrees of edges, \textit{J. Math. Chem.} \textbf{46}, (2009) 1369-1376.

\bibitem{z} Z. Yarahmadi, A. R. Ashrafi, S. Moradi, Extremal polyomino chains with respect to Zagreb indices, \textit{Appl. Math. Lett.} \textbf{25}, (2012) 166-171.

\bibitem{z2} B. Zhou, N. Trinajsti\'{c}, On a novel connectivity index, \textit{J. Math. Chem.} \textbf{46}, (2009) 1252-1270.





\end{thebibliography}
\end{document}